\documentclass[reqno]{amsart}\usepackage{amsmath}

\usepackage{amsfonts}
\usepackage{amssymb}
\usepackage{hyperref}

\begin{document}
\title[Existence of positive solutions]{A note on existence results for a nonlinear fourth-order integral boundary value problem}
\author[F. Haddouchi]{Faouzi Haddouchi}

\address{F. H:
Department of Physics, University of Sciences and Technology of Oran-MB, El Mnaouar, BP 1505, 31000 Oran, Algeria.}
\address{Laboratory of Fundamental and Applied Mathematics of Oran (Lmfao), Department of Mathematics, University of Oran 1 Ahmed Benbella, 31000 Oran,
Algeria}
\email{fhaddouchi@gmail.com}

\begin{abstract}
In this short note, we present some new existence results for a nonlinear
fourth-order two-point boundary value problem with integral condition. The
existence results are obtained by using the Leray-Schauder fixed point
theorem. Our work improves the main results of Benaicha and Haddouchi \cite%
{Benai}. In addition, examples are included to show the validity of our results.
\end{abstract}

\maketitle \numberwithin{equation}{section}
\newtheorem{theorem}{Theorem}[section]
\newtheorem{lemma}[theorem]{Lemma} \newtheorem{proposition}[theorem]{%
Proposition} \newtheorem{corollary}[theorem]{Corollary} \newtheorem{remark}[%
theorem]{Remark}
\newtheorem{exmp}{Example}[section]
\section{Introduction}
Fourth-order ordinary differential equations are models for bending or
deformation of elastic beams, and therefore have important applications in
engineering and physical sciences. Recently, the two-point and multi-point
boundary value problems for fourth-order nonlinear differential equations
have received much attention from many authors. Many authors have studied
the beam equation under various boundary conditions and by different
approaches. We refer the readers to the papers \cite{Alves, Graef1, Graef2,
Hend, Kos, Ma1, Ma2, Ping, Webb, Yao, Yang, Zhang, Graef3, Ander, Han, Sun,
Shen, Benai}.

In 2016, Benaicha and Haddouchi \cite{Benai}, by applying the
Krasnoselskii's fixed point theorem in cones, established the existence of
positive solutions for the following fourth-order two-point boundary value
problem (BVP) with integral boundary condition:

\begin{equation}  \label{eq-1.1}
{u^{\prime \prime \prime \prime }}(t)+f(u(t))=0,\ t\in(0,1),
\end{equation}
\begin{equation}  \label{eq-1.2}
u^{\prime}(0)=u^{\prime}(1)=u^{\prime \prime}(0)=0, \
u(0)=\int_{0}^{1}a(s)u(s)ds,
\end{equation}
where

\begin{itemize}
\item[(H1)] $f\in C([0,\infty),[0,\infty))$;

\item[(H2)] $a\in C([0,1],[0,\infty))$ and $0<\int_{0}^{1}a(s)ds<1$.
\end{itemize}

To obtain the existence of at least one positive solutions for this problem,
they assumed that the nonlinear term $f$ is either superlinear or sublinear.
That is, defining
\begin{equation*}
f_{0}=\lim_{u\to 0+}\frac{f(u)}{u}, \ \ f_{\infty}=\lim_{u\to\infty}\frac{%
f(u)}{u},
\end{equation*}

then, $f_{0}=0$ and $f_{\infty}=\infty$ correspond to the superlinear case,
and $f_{0}=\infty$ and $f_{\infty}=0$ correspond to the sublinear case.

This note applies the Leray-Schauder fixed point theorem to eliminate half
of the assumptions to prove the existence of a solution when using Krasnoselskii's fixed point
theorem of norm type with super and sub linear hypotheses.

Motivated by the work mentioned above, the aim of this note is to improve
the results in \cite{Benai} by showing that the BVP \eqref{eq-1.1} and %
\eqref{eq-1.2} has at least a positive solution if $f_{0}=0$ (condition $%
f_{\infty}=\infty$ being unnecessary), as well as, for $f_{\infty}=0$
(condition $f_{0}=\infty$ being also unnecessary).

For our analysis we use the Leray-Schauder's fixed point theorem.

\begin{lemma}[\cite{Gran}]
(Leray-Schauder) \label{lem 1.1} Let $\Omega$ be a convex subset in a Banach
space $X$, $0\in \Omega$ and assume that $A : \Omega \rightarrow \Omega$ is
a completely continuous operator. Then, either (i) $A$ has at least one
fixed point in $\Omega$; or (ii) the set $\{x\in \Omega / x=\lambda Ax,
0<\lambda<1\}$ is unbounded.
\end{lemma}

\section{Some preliminary results}
In order to prove our main results, we need some preliminary results.
Consider the following two-point boundary value problem

\begin{equation}  \label{eq-2.1}
{u^{\prime \prime \prime \prime }}(t)+y(t)=0,\ t\in(0,1),
\end{equation}
\begin{equation}  \label{eq-2.2}
u^{\prime}(0)=u^{\prime}(1)=u^{\prime \prime}(0)=0, \
u(0)=\int_{0}^{1}a(s)u(s)ds.
\end{equation}

For problem \eqref{eq-2.1}, \eqref{eq-2.2}, we have the following
conclusions which are derived from \cite{Benai}.

\begin{lemma}\emph{(\cite[Lemma 2.2]{Benai})}
\label{lem 2.1} The problem \eqref{eq-2.1}-\eqref{eq-2.2} has a unique
solution
\begin{equation*}
u(t)=\int_{0}^{1}\Big(G(t, s)+\frac{1}{1-\alpha}\int_{0}^{1}a(\tau)G(\tau,
s)d\tau \Big)y(s)ds,
\end{equation*}
where $G(t, s):[0, 1]\times[0, 1]\rightarrow \mathbb{R}$ is the Green's
function defined by
\begin{equation}  \label{eq-2.3}
G(t, s)=\frac{1}{6}%
\begin{cases}
t^{3}(1-s)^{2}-(t-s)^{3}, & 0\leq s\leq t \leq1; \\
t^{3}(1-s)^{2}, & 0\leq t \leq s\leq 1,%
\end{cases}%
\end{equation}
and
\begin{equation*}
\alpha=\int_{0}^{1}a(t)dt.
\end{equation*}
\end{lemma}

\begin{lemma}\emph{(\cite[Lemma 2.3]{Benai})}
\label{lem 2.2} Let $\theta\in]0, \frac{1}{2}[$ be fixed. Then
\begin{itemize}
\item[(i)] $G(t, s)\geq0$, for all $t, s\in[0, 1];$
\item[(ii)] $\frac{1}{6}\theta^{3}s(1-s)^{2}\leq G(t, s)\leq \frac{1}{6}%
s(1-s)^{2}$, for all $(t, s)\in[\theta, 1-\theta]\times[0, 1]$.
\end{itemize}
\end{lemma}

\begin{lemma}\emph{(\cite[Lemma 2.4]{Benai})}
\label{lem 2.3} Let $y(t)\in C([0, 1], [0, \infty))$ and $\theta\in]0, \frac{%
1}{2}[$. The unique solution of \eqref{eq-2.1}-\eqref{eq-2.2} is nonnegative
and satisfies
\begin{equation*}
\min_{t\in[\theta,1-\theta]}u(t)\geq \theta^{3}(1-\alpha+\beta) \|u\|,
\end{equation*}
where $\beta=\int_{\theta}^{1-\theta}a(t)dt$,\ \ $\alpha=\int_{0}^{1}a(t)dt$.
\end{lemma}

\section{Existence results}

In this section, we will state and prove our main results.

Let $\theta\in]0, \frac{1}{2}[$, $X=C([0, 1], \mathbb{R})$, $%
\beta=\int_{\theta}^{1-\theta}a(t)dt$, $\alpha=\int_{0}^{1}a(t)dt$, and
define the cone
\begin{equation*}
K=\left\{u\in X, u\geq0: \min_{t\in[\theta,1-\theta]}u(t)\geq
\theta^{3}(1-\alpha+\beta) \|u\|\right\}.
\end{equation*}
From lemmas \eqref{eq-2.1}, \eqref{eq-2.2}, and \eqref{eq-2.3}, the function
$u$ is a positive solution of the boundary value problem \eqref{eq-1.1} and %
\eqref{eq-1.2} if and only if $u(t)$ is a fixed point of the operator

\begin{equation}  \label{eq-3.1}
Au(t):=\int_{0}^{1}\Big(G(t, s)+\frac{1}{1-\alpha}\int_{0}^{1}a(\tau)G(\tau,
s)d\tau\Big)f(u(s))ds.
\end{equation}

\begin{theorem}\label{thm 3.1}
 Assume that $f_{0}=0$. Then BVP \eqref{eq-1.1} and %
\eqref{eq-1.2} has at least one positive solution.
\end{theorem}

\begin{proof}

Since $f_{0}=0$, there exists $\rho_{1}>0$ such that $f(u)\leq \epsilon u$,
for $0<u\leq \rho_{1}$, where $\epsilon>0$ satisfies
\begin{equation*}
{\epsilon}\leq {1-\alpha}.
\end{equation*}

If we denote

\begin{equation*}
\Omega=\Big\{u\in K,\ \|u\|\leq\rho_{1}\Big\},
\end{equation*}
Then $\Omega$ is a convex subset of $X$.

For $u\in\Omega$, according to the proofs of lemmas \ref{lem 2.2} and \ref%
{lem 2.3}, we have
\begin{gather}  \label{eq-3.2}
\begin{aligned} Au(t)&=\int_{0}^{1}\Big(G(t,
s)+\frac{1}{1-\alpha}\int_{0}^{1}a(\tau)G(\tau, s)d\tau\Big)f(u(s))ds\\
&\leq\int_{0}^{1}\Big(g(s)+\frac{1}{1-\alpha}\int_{0}^{1}g(s)a(\tau)d\tau%
\Big)f(u(s))ds\\ &=\frac{1}{1-\alpha}\int_{0}^{1}g(s)f(u(s))ds,\ \
t\in[0,1], \end{aligned}
\end{gather}
where $g(s)=\frac{1}{6}s(1-s)^{2}$.

So,
\begin{equation}  \label{eq-3.3}
\|Au\|\leq \frac{1}{1-\alpha}\int_{0}^{1}g(s)f(u(s))ds.
\end{equation}
In view of lemma \ref{lem 2.3} and \eqref{eq-3.3}, we have $Au(t)\geq0$ and

\begin{gather}
\begin{aligned} Au(t)&=\int_{0}^{1}\Big(G(t,
s)+\frac{1}{1-\alpha}\int_{0}^{1}a(\tau)G(\tau, s)d\tau\Big)f(u(s))ds\\
&\geq\theta^{3}\int_{0}^{1}\Big(g(s)+\frac{1}{1-\alpha}\int_{\theta}^{1-%
\theta}g(s)a(\tau)d\tau\Big)f(u(s))ds\\
&=\theta^{3}\frac{1-\alpha+\beta}{1-\alpha}\int_{0}^{1}g(s)f(u(s))ds\\
&\geq\theta^{3}(1-\alpha+\beta)\|Au\|, \ \ t\in[\theta, 1-\theta].
\end{aligned}  \label{eq-3.4}
\end{gather}
Hence,
\begin{equation*}
\min_{t\in[\theta,1-\theta]}Au(t)\geq \theta^{3}(1-\alpha+\beta) \|Au\|.
\end{equation*}
On the other hand,
\begin{gather}
\begin{aligned} Au(t)&\leq\frac{1}{1-\alpha}\int_{0}^{1}g(s)f(u(s))ds\\
&\leq\frac{1}{1-\alpha}\epsilon\|u\|\int_{0}^{1}g(s)ds\\
&\leq\frac{1}{6(1-\alpha)}\epsilon\|u\|\\ &\leq \|u\|\leq \rho_{1}.
\end{aligned}  \label{eq-3.5}
\end{gather}
Thus, $\|Au\|\leq\rho_{1}$. Hence $A\Omega\subset\Omega$. $%
A:\Omega\rightarrow\Omega$ is completely continuous by an application of
Arzela-Ascoli theorem.
\end{proof}

For $u\in \mathcal{V}$ with
\begin{equation*}
\mathcal{V}=\Big\{u\in\Omega / \ u=\lambda Au,\ 0<\lambda<1\Big\},
\end{equation*}
we have
\begin{equation*}
u(t)=\lambda Au(t)< Au(t)\leq \rho_{1},
\end{equation*}

which implies that $\|u\|\leq\rho_{1}$.

So, $\mathcal{V}$ is bounded. By Lemma \ref{lem 1.1}, the operator $A$ has at
least one fixed point in $\Omega$, which is a positive solution of %
\eqref{eq-1.1} and \eqref{eq-1.2}.

\begin{theorem}\label{thm 3.2}
If $f_{\infty}=0$, then BVP \eqref{eq-1.1} and \eqref{eq-1.2}
has at least one positive solution.
\end{theorem}

\begin{proof}
We discuss two possible cases:\

Case 1. If $f$ is bounded. Then, there exists $L>0$ such that $f(u)\leq L$.

For $u\in K$, we have $Au\in K$, and $A$ is completely continuous. Similar
to the estimates of \eqref{eq-3.2}, we obtain

\begin{gather}  \label{eq-3.6}
\begin{aligned} Au(t)&\leq\frac{1}{1-\alpha}\int_{0}^{1}g(s)f(u(s))ds\\
&\leq\frac{L}{1-\alpha}\int_{0}^{1}g(s)ds\\ &\leq\frac{L}{6(1-\alpha)}.
\end{aligned}
\end{gather}
Thus, $\|Au\|\leq \frac{L}{6(1-\alpha)}$. For $u\in \mathcal{V}$ with
\begin{equation*}
\mathcal{V}=\Big\{u\in K / \ u=\lambda Au,\ 0<\lambda<1\Big\},
\end{equation*}
we have
\begin{equation*}
u(t)=\lambda Au(t)< Au(t)\leq \frac{L}{6(1-\alpha)},
\end{equation*}

which implies that $\|u\|\leq \frac{L}{6(1-\alpha)}$.

So, $\mathcal{V}$ is bounded. By Lemma \ref{lem 1.1}, the operator $A$ has at
least one fixed point in $K$, which is a positive solution of \eqref{eq-1.1}
and \eqref{eq-1.2}.

Case 2. Suppose that $f$ is unbounded, since $f_{\infty}=0$, there exists ${%
\rho}_{2}>0$ such that $f(u)\leq \eta u$ for $u>{\rho}_{2}$, where $\eta>0$
satisfies
\begin{equation*}
{\eta}\leq 1-\alpha.
\end{equation*}
On the other hand, from condition \textrm{(H1)}, there is $\sigma>0$ such
that $f(u)\leq \eta\sigma$,\ \ with $0\leq u\leq {\rho}_{2}$.

Now, set
\begin{equation*}
\Omega=\Big\{u\in K,\ \|u\|\leq \widehat{\rho}_{2}\Big\},
\end{equation*}
where $\widehat{\rho}_{2}=\max\{\sigma,\rho_{2}\}$

If $u\in \Omega$, then we have $f(u)\leq\eta \widehat{\rho}_{2}$. Similar to %
\eqref{eq-3.2}, we have
\begin{gather}  \label{eq-3.7}
\begin{aligned} Au(t)&\leq\frac{1}{1-\alpha}\int_{0}^{1}g(s)f(u(s))ds\\
&\leq\frac{1}{6}\cdot\frac{\eta\widehat{\rho}_{2} }{1-\alpha}\\
&\leq\widehat{\rho}_{2}. \end{aligned}
\end{gather}
Thus, $\|Au\|\leq \widehat{\rho}_{2}$.

It is easy to check that $\mathcal{V}=\Big\{u\in\Omega / u=\lambda Au,\
0<\lambda<1\Big\}$ is bounded. Therefore, by Lemma \ref{lem 1.1}, the
boundary value problem \eqref{eq-1.1} and \eqref{eq-1.2} has at least one
positive solution.
\end{proof}

\section{Examples}

\begin{exmp}
Consider the fourth-order boundary value

\begin{equation}  \label{eq-4.1}
{u^{\prime \prime \prime \prime }}(t)+u(1-e^{-u})=0, \ \ 0<t<1,
\end{equation}
\begin{equation}  \label{eq-4.2}
u^{\prime}(0)=u^{\prime}(1)=u^{\prime \prime}(0)=0, \
u(0)=\int_{0}^{1}s^{2}u(s)ds,
\end{equation}
where $f(u)=u(1-e^{-u})\in C([0,\infty),[0,\infty))$ and $a(t)=t^{2}\geq0$, $%
\int_{0}^{1}a(s)ds=\int_{0}^{1}s^{2}ds=\frac{1}{3}$.

We have
\begin{equation*}
f_{0}=\lim_{u\to 0+}\frac{f(u)}{u}=\lim_{u\to 0+}(1-e^{-u})=0.
\end{equation*}
Thus, it follows from Theorem \ref{thm 3.1} that the problem \eqref{eq-4.1}
and \eqref{eq-4.2} has at least one positive solution.
Notice that $f_{\infty}=1$, $f_{\infty}\neq\infty$, so Theorem 3.1 in \cite%
{Benai} cannot be applied to show the existence of positive solutions for
the problem \eqref{eq-4.1} and \eqref{eq-4.2}.
\end{exmp}

\begin{exmp}
Consider the fourth-order boundary value

\begin{equation}  \label{eq-4.3}
{u^{\prime \prime \prime \prime }}(t)+1-e^{-u}=0, \ \ 0<t<1,
\end{equation}
\begin{equation}  \label{eq-4.4}
u^{\prime}(0)=u^{\prime}(1)=u^{\prime \prime}(0)=0, \
u(0)=\int_{0}^{1}s^{2}u(s)ds,
\end{equation}
where $f(u)=1-e^{-u}\in C([0,\infty),[0,\infty))$ and $a(t)=t^{2}\geq0$, $%
\int_{0}^{1}a(s)ds=\int_{0}^{1}s^{2}ds=\frac{1}{3}$.

Since
\begin{equation*}
f_{\infty}=\lim_{u\to+\infty}\frac{f(u)}{u}=\lim_{u\to+\infty}\frac{1-e^{-u}%
}{u}=0,
\end{equation*}
From Theorem \ref{thm 3.2}, the problem \eqref{eq-4.3} and \eqref{eq-4.4}
has at least one positive solution.

On the other hand, we have
\begin{equation*}
f_{0}=\lim_{u\to +0}\frac{f(u)}{u}=\lim_{u\to +0}\frac{1-e^{-u}}{u}=1.
\end{equation*}
therefore Theorem 3.2 in \cite{Benai} also cannot be applied to show the
existence of positive solutions for the problem \eqref{eq-4.3} and \eqref{eq-4.4}.
\end{exmp}


\end{document}